\documentclass[11pt,a4paper]{amsart}

\usepackage[T1]{fontenc}
\usepackage[utf8]{inputenc}
\usepackage{lmodern}
\usepackage[margin=1.08in]{geometry}
\usepackage{amsmath,amssymb,amsthm,mathtools}
\usepackage{booktabs}
\usepackage{enumitem}
\usepackage{microtype}
\usepackage{xcolor}
\usepackage{listings}
\usepackage{hyperref}
\usepackage{mathrsfs}
\newtheorem{theorem}{Theorem}[section]
\newtheorem{proposition}[theorem]{Proposition}
\newtheorem{lemma}[theorem]{Lemma}

\theoremstyle{definition}

\usepackage{tikz-cd}
\newcommand{\Q}{\mathbf Q}

\newcommand{\Gal}{\operatorname{Gal}}
\newcommand{\Br}{\operatorname{Br}}
\newcommand{\Pic}{\operatorname{Pic}}

\newcommand{\inv}{\operatorname{inv}}

\numberwithin{equation}{section}

\lstdefinestyle{magma}{
  basicstyle=\ttfamily\footnotesize,
  keywordstyle=\color{blue!55!black},
  commentstyle=\color{green!35!black},
  stringstyle=\color{red!50!black},
  columns=fullflexible,
  keepspaces=true,
  breaklines=true,
  frame=single,
  showstringspaces=false,
  aboveskip=1em,
  belowskip=1em,
}

\title{Counterexamples to O'Neil's Period-Index Problem on Elliptic Curves}

\author{Xiaoguang Shang}
\address{Department of Mathematics \\ Nanjing University \\ 22 Hankou Road Nanjing 210093 China}
\email{xgshang@smail.nju.edu.cn}

\author{Cheng Niu}
\address{Department of Mathematics \\ Nanjing University \\ 22 Hankou Road Nanjing 210093 China}
\email{chengniu@smail.nju.edu.cn}

\begin{document}
\maketitle

\begin{abstract}
Let $E/K$ be an elliptic curve over a number field $K$, and let $[C]\in H^1(K,E(\overline{K}))$ be a homogeneous space under $E$. Suppose that $C$ has period $n$ and index $d$.  O'Neil asked whether one can always choose a
lift of $[C]$ to $H^1(K,E[n])$ whose period-index obstruction has  order exactly
$d/n$.  We give a negative answer to this question by constructing an explicit family of examples.  Taking $K=\mathbb{Q}(\zeta_8)$, we prove that there exist infinitely many pairwise non-isomorphic elliptic
curves $E/K$ such that each $E$ admits infinitely many homogeneous spaces
$[C]\in H^1(K,E(\overline{K}))$ with period $8$ and index $16$, but the period-index
obstruction of every lift of $[C]$ to $H^1(K,E[8])$ has  order exactly $8$.
\end{abstract}

\section{Introduction}

Let $K$ be a field with absolute Galois group $G_K=\Gal(\overline{K}/K)$, and let $E/K$ be an elliptic curve over  $K$.  The Weil-Ch\^{a}telet group $\operatorname{WC}(E/K)$, which classifies the equivalence classes of homogeneous spaces of $E$, is canonically isomorphic to the Galois cohomology $H^1(K,E(\overline{K}))$ \cite[Proposition 4]{LangTate}. Under this isomorphism, each equivalence class of a homogeneous space $C$ of $E$ corresponds to a cohomology class in $H^1(K,E(\overline{K}))$, which we denote by $[C]$. The period of a homogeneous space $C$, denoted by $P(C)$, is defined as  the order of $[C]$ in the group $H^1(K,E(\overline{K}))$. The index of $C$, denoted by $I(C)$, is the smallest positive integer $d$ for which there exists a  $K$-rational divisor of degree $d$ on $C$.

It is well known that \(P(C)\) divides \(I(C)\) and that \(P(C)\) and \(I(C)\) have the same prime factors \cite[Proposition 5]{LangTate}. When \(K\) is a local field, Lichtenbaum \cite{Lichtenbaum} proved that \(P(C)=I(C)\) for every homogeneous space \(C\) of \(E/K\). In contrast, when \(K\) is a number field, the period and index do not necessarily coincide \cite{cassels}. In this case, the precise relationship between \(P(C)\) and \(I(C)\) remains mysterious \cite{ClarkSharif,ONeil,sharif1}. One approach to understanding this relationship is via O'Neil's period-index obstruction map. We  recall the following Kummer sequence for \(E/K\):
\begin{equation*}
	0\rightarrow E(K)/nE(K)\xrightarrow{\delta} H^1(K,E[n])\xrightarrow{\rho} H^1(K,E(\overline{K}))[n]\rightarrow 0.
\end{equation*}
O'Neil \cite{ONeil} introduced the period-index obstruction map
\begin{equation*}
	Ob_n: H^1(K,E[n])\rightarrow \Br(K),
\end{equation*}
and proved that, if \([C]\in H^1(K,E(\overline{K}))[n]\) satisfies \(P(C)=n\), then for any lift \(\xi\in H^1(K,E[n])\) of \([C]\), the index of \(C\) satisfies
\begin{equation}\label{equ}
	I(C)\leq P(C)\cdot \operatorname{order}(Ob_n(\xi)).
\end{equation}
In particular, we obtain
\begin{equation}\label{o2}
	\frac{I(C)}{P(C)}
	\leq
	\min_{\rho(\xi)=[C]} \operatorname{order}(Ob_n(\xi)).
\end{equation}
O'Neil \cite{ONeil} raised the question of whether the inequality \eqref{o2} is always an equality. Equivalently, if \(I(C)=d\), does there exist a lift \(\eta\in H^1(K,E[n])\) of \([C]\) such that the period-index obstruction \(Ob_n(\eta)\) has order \(d/n\) ?
Our main result gives a negative answer to this question by constructing an explicit family of examples.
\begin{theorem}\label{mainthm}
	Let $K=\mathbb{Q}(\zeta_8)$.
	Then there exists an infinite set \(\Omega\subset\mathbb{Q}\) and elliptic curves
	\(E_\lambda/\mathbb{Q}\), with pairwise distinct \(j\)-invariants for
	\(\lambda\in\Omega\), such that for every \(\lambda\in\Omega\) there are
	infinitely many homogeneous spaces $	[C]\in H^1(K,E_\lambda)$	satisfying
	\[
 P(C)=8,\qquad I(C)=16.
	\]
	Moreover, for every lift	$\eta\in H^1(K,E_\lambda[8])$	of \([C]\), the period-index obstruction class
	\(Ob_8(\eta)\) has  order exactly \(8\).
\end{theorem}
These curves in Theorem 1.1 arise from the Tate--Kubert family \cite{Kubert,x18} with a
marked point \(P\) of order \(8\).

We briefly explain the idea behind the construction of the homogeneous spaces. We construct infinitely many special cohomology classes
$
\eta\in H^1(K,E[8])
$
to control the period-index obstruction map. More precisely, the construction is arranged so that the Tate pairing \(T(P,\eta)\) has order exactly \(4\). Another key property is the identity
\begin{equation*}
2Ob_8(\eta)=T(P,\eta).
\end{equation*}
This gives a \(K\)-rational divisor of degree \(16\), and hence $I(C)\mid16$. Moreover, we use a method similar to that of   \cite{ClarkSharif} to ensure that every lift \(\eta\in H^1(K,E[8])\) of \([C]\) has period-index obstruction \(\operatorname{Ob}_8(\eta)\) of order exactly \(8\). This implies that the homogeneous space \(C\) has period exactly \(8\) and index exactly \(16\) (cf. Theorem \ref{22}).

\subsection{Organization of the paper}
The paper is organized as follows.

In Section 2, we introduce the definition and basic properties of the
period-index obstruction map and establish the necessary properties of
the Tate--Kubert family.

In Section 3, we construct cohomology classes satisfying certain
specified conditions and use them to construct the homogeneous space
\(C\).

In Section 4, we prove that the homogeneous space \(C\) satisfies the
conditions of Theorem \ref{mainthm}.

\subsection{Acknowledgments}
The counterexample presented in this paper was discovered with the assistance of OpenAI's
GPT-5.6 Sol model. All mathematical arguments and references were independently
verified by the authors.

\section{On the period-index obstruction map}
\subsection{Period-index obstruction map and Hilbert symbol}
Our main tool for computing the index is O'Neil's period-index obstruction map. In order to define it, we 
first recall the basic notion of theta group. For a positive integer $n$ not divisible by the characteristic of $K$, let $E[n]$ be the $n$-torsion points of the elliptic curve $E$.  Let $L=\mathcal{O}(nO_E)$ be the line bundle corresponding to the divisor $nO_E$, where $O_E$ is the identity of $E$. The theta group $\mathscr{G}_L$ is defined to be the set of pairs $(\varphi,\tau)$, where $\tau: E\rightarrow E$ is a translation and $\varphi $ is an isomorphism $\tau^* L\rightarrow L$. It is a group via
\[
(\varphi,\tau)\cdot(\varphi',\tau')=(\tau'^*(\varphi)\circ\varphi',\tau\circ\tau').
\]
Note that for our choice of $L$, $\tau$ must be a translation by an element of $E[n]$. Thus we have the following  exact sequence of $K$-group schemes
\[
0\rightarrow\mathbb{G}_m\rightarrow \mathscr{G}_L\rightarrow E[n]\rightarrow 0
\]
The subgroup $\mathbb{G}_m$ corresponds to the set $(a,id)$, where $id$ is the identity map and $a$ denotes multiplication by the constant $a$. The quotient map to $E[n]$ sends $(\varphi,\tau)$ to $\tau(O)$.

We consider the \(\overline{K}\)-points of the group schemes and take
(non-abelian) Galois cohomology. This gives rise to a coboundary map
\[
Ob_n:H^1(K,E[n])\longrightarrow H^2(K,\mathbb{G}_m)[n]=\operatorname{Br}(K)[n].
\]
This map is called the period-index obstruction map. In what follows, we
omit the subscript \(n\) from \(Ob_n\) when it is clear from the context.

The period-index obstruction map satisfies the following functoriality.
\begin{theorem}\label{cd}
	Let $n$ and $m$ be positive integers. The following diagram commutes:
	\begin{equation}
		\begin{tikzcd}
		H^1(K,E[n])\ar[r,"Ob_n"]\ar[d,"j_*"]	&\Br(K)\ar[d,"m"] \\
		H^1(K,E[mn])\ar[r,"Ob_{mn}"]	& \Br (K)
		\end{tikzcd}
	\end{equation}
where $j_*$ is induced by the natural inclusion $j:E[n]\rightarrow E[mn] $, and $m$ is multiplication by $m$.
\end{theorem}

We will also need another equivalent definition. Let $C$ be a homogeneous space of $E$ and let $\operatorname{Pic}(C)$ be the Picard group of $C$. There is an exact sequence \cite[Section 3.1]{clark}
\begin{equation}\label{es}
	0\rightarrow \operatorname{Pic}(C)\rightarrow H^0(K,\operatorname{Pic}(C_{\overline{K}}))\xrightarrow{\delta_C} \Br(K)
\end{equation}
On the other hand, a Galois descent argument \cite[Proposition 1.9]{CFOSS1} shows that $H^1(K,E[n])$ classifies pairs $(C,D)$, where $[C]\in H^1(K,E(\overline{K}))$ and $D\in \Pic^n(C)(K)$ is a $K$-rational divisor class of degree $n$. Then the argument in  \cite[Section 3.1]{clark} shows
\[
Ob_n(C,D)=\delta_C([D]).
\]
This equality gives an explanation of the relationship between the obstruction map and the index. In \cite{ONeil}, O'Neil showed
\begin{proposition}\label{neil}
	 Let $C$ be  a homogeneous space of $E$ with period dividing $n$ and let $\xi\in H^1(K,E[n])$ be any lift of $C$. If $Ob_n(\xi)$ has order $l$, then the index of $C$ divides $nl$.
\end{proposition}
  For later use, we recall O'Neil's argument. If  $Ob_n(C,D)=\delta_C([D])$ is trivial, then by exact sequence \eqref{es}, the $K$-rational divisor class $[D]$ lifts to an element of $\Pic(C)$, yielding a $K$-rational divisor $D$ of degree $n$ on $C$. More generally, suppose that the 
\[
\operatorname{order}(Ob_n(C,D))=\operatorname{order}(\delta_C([D]))=l.
\]
By the theory of Brauer--Severi varieties, there exists a field extension
\(F/K\) of degree \(l\) that splits \(\delta_C([D])\). Hence, the restriction
of \(Ob_n(C,D)\) to \(F\) is trivial. This implies that \(C\) admits an
\(F\)-rational divisor \(D_F\) of degree \(n\). Taking the sum of the $l$-conjugates of $D_F$ under $\operatorname{Hom}_K(F,\overline{K})$, we obtain a $K$-rational divisor of degree $nl$ on $C$. This shows that the index of $C$ divides $n\cdot \operatorname{order}(Ob_n(C,D))$.
   
 We also need the following Tate pairing to compute the period-index
 obstruction map:
 \begin{equation}
 	T(-,-): H^1(K,E[n])\times H^1(K,E[n])
 	\xrightarrow{\cup}
 	H^2(K,E[n]\otimes E[n])
 	\xrightarrow{e_*}
 	H^2(K,\mu_n),
 \end{equation}
 where the map \(e_*\) is induced by the Weil pairing, and the latter
 group \(H^2(K,\mu_n)\) is canonically isomorphic to \(\Br(K)[n]\). The pairing \(T\) is bilinear and symmetric, since the Weil pairing \(e_n\) is alternating and interchanging the two degree-one cup-product factors introduces a second minus sign.
 
 Let $
 \delta_n:E(K)/nE(K)\longrightarrow H^1(K,E[n])$
 be the connecting homomorphism in the Kummer sequence of \(E\), and let $
 j_*:H^1(K,E[n])\longrightarrow H^1(K,E[mn])$
 be the map induced by the natural inclusion \(E[n]\hookrightarrow E[mn]\). We have the following compatibility of the Tate pairing.
 \begin{proposition}\label{tate2}
 	For every \(P\in E(K)\) and \(\xi\in H^1(K,E[n])\), we have
 	\[
 	T(\delta_n(P),\xi)
 	=
 	T(\delta_{mn}(P),j_*(\xi)).
 	\]
 \end{proposition}
\begin{proof}
Let \(\alpha\) and \(\beta\) be the corresponding \(2\)-cocycles
representing
$T(\delta_n(P),\xi)$ and $T(\delta_{mn}(P),j_*(\xi))$,
respectively. We continue to denote by
\(\xi\in Z^1(G_K,E[n])\) a \(1\)-cocycle representative of the
cohomology class \(\xi\).

Choose \(Q,R\in E(\overline{K})\) such that $nQ=P, mR=Q.$ Then \(mnR=P\). By the definition of the Tate pairing, for any \(\sigma,\tau\in G_K\), we have
\[
\begin{aligned}
	\beta(\sigma,\tau)
	&=e_{mn}(R^\sigma-R,\sigma(j_*(\xi(\tau))))\\
	&=e_{mn}(R^\sigma-R,\sigma(\xi(\tau)))\\
	&=e_n(mR^\sigma-mR,\sigma(\xi(\tau)))\\
	&=e_n(Q^\sigma-Q,\sigma(\xi(\tau)))\\
	&=\alpha(\sigma,\tau).
\end{aligned}
\]
\end{proof}
 We adopt the following convention: whenever \(\xi=\delta(P)\) for some \(P\in E(K)\), we denote \(T(\delta(P),\eta)\) simply by \(T(P,\eta)\).
\begin{proposition}\label{ff}\cite[Proposition 4.1, Proposition 4.3]{ONeil}
 The period-index obstruction map $Ob$ is quadratic, that is,
 $
  Ob(a\xi)=a^2Ob(\xi)
  $
  for $a\in \mathbb{Z}$, and, for all $\xi,\eta \in H^1(K,E[n])$, 
	\begin{equation}
		T(\xi,\eta)=Ob(\xi+\eta)-Ob(\xi)-Ob(\eta).
	\end{equation}
In particular, 
$
T(\xi,\xi)=2Ob(\xi).
$
\end{proposition}

From now on, we assume that $K$ contains a primitive
$n$-th root of unity $\zeta_n$ and that $E/K$ is an elliptic curve with a
$K$-rational point $P\in E[n](K)$ of order exactly $n$. We identify
$\langle P\rangle$ with $\mathbb{Z}/n\mathbb{Z}$. By the Kummer
sequence, we have canonical isomorphisms
\[
K^*/K^{*n}\cong H^1(K,\langle P\rangle)\cong H^1(K,\mathbb{Z}/n\mathbb{Z})
\]
and
\[
H^2(K,\langle P\rangle)\cong H^2(K,\mathbb{Z}/n\mathbb{Z})\cong Br(K)[n].
\]
Taking the Weil pairing with $P$ gives the following short exact
sequence:
\begin{align}\label{aa}
	0\longrightarrow \langle P\rangle \longrightarrow E[n]
	&\longrightarrow \mu_n\longrightarrow 0,\\
	Q&\longmapsto e(P,Q).
\end{align}
Passing to Galois cohomology, we obtain the exact sequence 
\begin{equation}\label{bb}
 H^1(K,\mathbb{Z}/n\mathbb{Z})\xrightarrow{i_*} H^1(K,E[n])\xrightarrow{\pi_*} H^1(K,\mu_n)\cong K^*/K^{*n}\xrightarrow{\delta} H^2(K,\mathbb{Z}/n\mathbb{Z})\cong \Br(K)[n].
	\end{equation}	

We first recall the Hilbert symbol, and then establish its relationship with the Tate pairing.
The natural pairing  
\begin{align*}
	\mathbb{Z}/n\mathbb{Z}\times \mu_n\rightarrow \mu_n,\\
	 (a,\zeta)\mapsto \zeta^a,
\end{align*}
induces the following cup product pairing:
\begin{equation*}
	H^1(K,\mathbb{Z}/n\mathbb{Z})\times H^1(K,\mu_n)\rightarrow H^2(K,\mathbb{Z}/n\mathbb{Z}\otimes \mu_n)\rightarrow H^2(K,\mu_n).
\end{equation*}
We refer to this pairing as the Hilbert symbol, denoted by $[-,-]$ (see \cite[Chapter XIV]{Serre}). 
\begin{proposition}\label{tate}
	For any character $\chi\in H^1(K,\mathbb{Z}/n\mathbb{Z})$ and $\xi \in H^1(K,E[n])$, we have 
	\begin{equation}
		T(i_*(\chi),\xi)=[\chi,\pi_*(\xi)].
	\end{equation}
\end{proposition}
\begin{proof}
	Let $\chi\in Z^1(K,\mathbb{Z}/n\mathbb{Z})$ and
	$\xi\in Z^1(K,E[n])$ be $1$-cocycle representatives of the
	corresponding cohomology classes. By the definition of the maps in the
	exact sequence \eqref{aa}, we have
	$
	i_*(\chi)(\sigma)=\chi(\sigma)P
	$
	and
	$
	\pi_*(\xi)(\tau)=e(P,\xi(\tau)).
	$
	Let $\alpha$ and $\beta$ be the corresponding $2$-cocycles
	representing $T(i_*(\chi),\xi)$ and $[\chi,\pi_*(\xi)]$,
	respectively. Then
	\begin{align*}
		\alpha(\sigma,\tau)
		&=e(i_*(\chi)(\sigma),\sigma(\xi(\tau)))\\
		&=e(\chi(\sigma)P,\sigma(\xi(\tau)))\\
		&=\sigma(e(P,\xi(\tau)))^{\chi(\sigma)}.
	\end{align*}
	On the other hand, by the definition of the Hilbert symbol, we have
	$
	\beta(\sigma,\tau)
	=\sigma(e(P,\xi(\tau)))^{\chi(\sigma)}.
	$

	Hence $\alpha=\beta$, and therefore $	T(i_*(\chi),\xi)=[\chi,\pi_*(\xi)].$
\end{proof}
 	Recall the following short exact sequence:
\begin{equation*}
	0\longrightarrow \langle P\rangle
	\longrightarrow E[n]
	\xrightarrow{e(P,-)}
	\mu_n
	\longrightarrow 0.
\end{equation*}
Let $c\in H^1(K,\mathbb{Z}/n\mathbb{Z})$ be the character associated with this exact sequence. More explicitly, choose a point $Q\in E[n]$ such that $e(P,Q)=\zeta_n$. Since $P$ and
$\zeta_n$ are $K$-rational, for every $\sigma\in G_K$ there exists a unique element $c_\sigma\in\mathbb{Z}/n\mathbb{Z}$ such that
$
\sigma(Q)=Q+c_\sigma P.
$
The character $c$ is defined by
\[
c:G_K\longrightarrow \mathbb{Z}/n\mathbb{Z},
\qquad
\sigma\longmapsto c_\sigma .
\]
The kernel of $c$ fixes both $P$ and $Q$, and hence fixes all of $E[n]$.
Therefore, the fixed field of $\ker(c)$ is precisely the division field
$K(E[n])$.
Now we compute the coboundary map $\delta: H^1(K,\mu_n)\rightarrow H^2(K,\mathbb{Z}/n\mathbb{Z})$ arising from  the exact sequence \eqref{bb}.
\begin{proposition}\label{connect}
	For any character $\chi \in H^1(K,\mu_n)$, we have $\delta(\chi)=[c,\chi]$.
\end{proposition}
\begin{proof}
 This result follows from a direct cocycle computation.
\end{proof}
\subsection{The Tate--Kubert family and its generic affine group}

Let $\lambda$ be a parameter and put
\begin{equation*}\label{eq:bc-lambda}
 b_\lambda=\frac{\lambda(\lambda-1)}{(\lambda+1)^2},
 \qquad
 c_\lambda=\frac{\lambda(\lambda-1)}{\lambda+1}.
\end{equation*}
Consider
\begin{equation*}\label{eq:Tate-family}
 E_\lambda:\quad
 y^2+(1-c_\lambda)xy-b_\lambda y
 =x^3-b_\lambda x^2,
 \qquad P=(0,0).
\end{equation*}
The point $P$ has exact order $8$ on every smooth fiber.  The
discriminant is
\begin{equation*}\label{eq:family-disc}
 \Delta_\lambda
 =\frac{\lambda^8(\lambda-1)^4
 (\lambda^2-6\lambda+1)}{(\lambda+1)^{10}}.
\end{equation*}
This is the
standard rational parameter on $X_1(8)$ (see \cite{AntieauAuel,Kubert,x18}).

Let $F_0=\mathbb{Q}(\lambda)$ and define
\[
L_{\mathrm{gen}}=F_0(E_\lambda[8]),
\qquad
H_{\mathrm{gen}}=L_{\mathrm{gen}}([8]^{-1}P).
\]
Fix a basis $(P_\lambda,Q_\lambda)$ of $E_\lambda[8]$. The action of
$\operatorname{Gal}(L_{\mathrm{gen}}/F_0)$ on this basis induces an injective
homomorphism
\begin{equation*}
 	\rho_{E_{\lambda, 8}} :\operatorname{Gal}(L_{\mathrm{gen}}/F_0)
	\hookrightarrow
	\operatorname{GL}_2(\mathbb{Z}/8\mathbb{Z}).
\end{equation*}
Since $P_\lambda$ is $F_0$-rational, the linear Galois action fixes the vector $P_\lambda$.
Hence, with respect to this basis, every matrix in the image has the form
\[
\begin{pmatrix}
	1&a\\
	0&u
\end{pmatrix},
\qquad
a\in\mathbb{Z}/8\mathbb{Z},
\qquad
u\in(\mathbb{Z}/8\mathbb{Z})^\times .
\]
The determinant of this representation is the cyclotomic character.
Therefore, the largest possible image of the linear Galois representation
over $\mathbb{Q}(\lambda)$ has order 32.

Choose and fix any $R$ satisfying $8R=P$. Then we also have an injective
homomorphism
\begin{align*}
	\operatorname{Gal}(H_{\mathrm{gen}}/L_{\mathrm{gen}})
	&\hookrightarrow E_\lambda[8],\\
	\sigma&\longmapsto R^\sigma-R .
\end{align*}
\begin{theorem}
	The Galois representation $\rho_{E_{\lambda, 8}}$ has maximal possible image. Namely,
	\[
	\operatorname{Im}(\rho_{E_\lambda,8})
	=
	\left\{
	\begin{pmatrix}
		1&a\\
		0&u
	\end{pmatrix}
	:
	a\in\mathbb Z/8\mathbb Z,\,
	u\in(\mathbb Z/8\mathbb Z)^\times
	\right\}.
	\]
	Moreover,  $\operatorname{Gal}(H_{\mathrm{gen}}/L_{\mathrm{gen}})$ is also maximal, that is
$
	\operatorname{Gal}(H_{\mathrm{gen}}/L_{\mathrm{gen}})
	\cong E_\lambda[8].
$
\end{theorem}
\begin{proof}
	It suffices to verify that there exists a special value
	$\lambda=\lambda_0\in\mathbb{Q}$ for which the associated Galois
	representation has maximal possible image. We will choose such a value of
	$\lambda$ in  Section 2.3.
\end{proof}

Now choose a 
primitive element $\theta$ for $H_{\mathrm{gen}}/\Q(\lambda)$ and let
$f(\lambda,X)$ be its irreducible polynomial, of degree $64\times 32=2048$.  After
removing the finitely many poles, discriminant zeros, and bad reduction of $E_\lambda$, Hilbert irreducibility
\cite[Chapter 3, Proposition 3.3.5]{SerreTG}
gives a Hilbert subset
$
\Omega\subset \mathbb{A}^1(\mathbb{Q})
$
such that $f(\lambda,X)$ remains irreducible of degree $2048$ for every
$\lambda\in\Omega$.  

Thus  for every $\lambda\in\Omega$, the  field $H_{gen}$ contains $K=\Q(\zeta_8)$, so its degree over
$K$ is $2048/4=512$.  Let $L_\lambda=K(E_\lambda[8])$ and $H_\lambda=L_\lambda(8^{-1}P)$, then 
$
[L_\lambda:K]\le 8$,
$
[H_\lambda:L_\lambda]\le64.
$
Their product is $512=8\cdot64$, forcing equality in both bounds.  Therefore
for every $\lambda\in\Omega$ one has
\begin{equation*}\label{eq:full-specialized}
	\Gal(K(E_\lambda[8])/K)\simeq\mathbb{Z}/8\mathbb{Z},
\end{equation*}
and
\begin{equation*}\label{eq:full-translation}
	\Gal\bigl(K(E_\lambda[8],[8]^{-1}P_\lambda)/K(E_\lambda[8])\bigr)
	=E_\lambda[8].
\end{equation*}

The $j$-map of the family is nonconstant: it is the forgetful finite map
$X_1(8)\to X(1)$. After passing to an infinite subset of $\Omega$, the curves
$E_\lambda$ therefore have pairwise distinct $j$-invariants.

\subsection{Exact specialization certificate}\label{sec:certificate}

Take $\lambda_0=-\frac{19}{17}$. The specialized curve is
\begin{equation}\label{eq:E0}
 E_0:\quad
 y^2+\frac{359}{17}xy-171y=x^3-171x^2,
 \qquad P=(0,0).
\end{equation}

\begin{proposition}\label{prop:certificate}
Let $K=\Q(\zeta_8)$.  For the curve \eqref{eq:E0}:
\begin{enumerate}[label=\textup{(\roman*)}]
\item
$
 L_0=K(E_0[8])=K(u),\qquad u^8=647,
 \qquad [L_0:K]=8.
$
\item \label{basis}There is a symplectic basis $(P,Q)$ with $e_8(P,Q)=\zeta_8$ such
that, for $\tau(u)=\zeta_8u$,
\[
 \tau(P)=P,\qquad \tau(Q)=Q+P,
\]
then $\Gal(L_0/K)=\langle \tau \rangle\simeq \mathbb{Z}/8\mathbb{Z}$.
\item If
$
 H_0=L_0([8]^{-1}P),
$
then\[
 \Gal(H_0/L_0)=E_0[8].
\]
\end{enumerate}
\end{proposition}

\begin{proof}
The statements in (i) and (ii) can be checked using the Magma code included in Appendix.\ref{app:magma}. 

For (iii), we use the rational prime \(281\). Since the elliptic curve \(E_0\) has good reduction at \(281\) and \(8\nmid 281\), the reduction map
\[
E_0[8]\longrightarrow \widetilde{E_0}[8]
\]
is injective \cite[Chapter VII, Proposition 3.1]{Sil}. Hence, by \ref{basis}, the reductions \(\widetilde{P},\widetilde{Q}\) form a basis of \(\widetilde{E_0}[8]\).

Since \(281\equiv 1\pmod{8}\) and \(647\) is an eighth power in \(\mathbb{F}_{281}\), every prime ideal of \(K\) lying above \(281\) splits completely in \(L_0\). Fix one such prime ideal and denote it by \(\mathfrak{p}\). We use the Magma code in Appendix \ref{magma2} to compute the action of the Frobenius element at \(\mathfrak{p}\) on \([8]^{-1}\widetilde{P}\). The exact finite-field computation shows that the Frobenius acts on \([8]^{-1}\widetilde{P}\) by translation by $5\widetilde{P}+\widetilde{Q}$. Moreover,  \(\tau\) acts on \(\Gal(H_0/L_0)\) by conjugation and sends $
5\widetilde{P}+\widetilde{Q}
$ to $ 6\widetilde{P}+\widetilde{Q}$. Since these two elements generate the whole group \(\widetilde{E}_0[8]\), it follows that $\Gal(H_0/L_0)=E_0[8]$.

\end{proof}

For the remainder of the proof, fix a value $\lambda\in\Omega$, and
abbreviate
\[
E=E_\lambda,\qquad P=P_\lambda,
\]
and
\[
K=\mathbb{Q}(\zeta_8),\qquad
L=K(E[8]),\qquad
H=L([8]^{-1}P).
\]
We may choose  and fix a symplectic basis
$(P,Q)$ of $E[8]$ satisfying
$
e_8(P,Q)=\zeta_8,
$
and fix a generator $\tau\in\Gal(L/K)$ such that
$
\tau(P)=P$, $ \tau(Q)=P+Q.
$

The construction must control the variation of the obstruction arising
from every class in $E(K)/8E(K)$, rather than only those generated by
$P$. This motivates the introduction of the following field.

Choose representatives
$
R_1,\ldots,R_s
$
of the finite group $E(K)/2E(K)$, and choose points
$S_i\in E(\overline{K})$ satisfying $2S_i=R_i$.
Define
$\label{eq:FMW}
	F_{\mathrm{MW}}=L(S_1,\ldots,S_s).
$
Over $L$, every automorphism of $F_{\mathrm{MW}}$ acts on each $S_i$
by translation by a point of $E[2]$. Hence $\Gal(F_{\mathrm{MW}}/L)$ is an elementary abelian $2$-group. Let $A=HF_{\mathrm{MW}}$ be the compositum of $H$ and $F_{\mathrm{MW}}$. In summary, we have the
following diagram:
\begin{equation}
	\begin{tikzcd}
		&A\ar[dl,no head]\ar[d,no head]\\
		H=L([8]^{-1}P)\ar[d,no head]&F_{\mathrm{MW}}\ar[dl,no head]\\
		L=K(E[8])\ar[d,no head]&\\
		K&
	\end{tikzcd}
\end{equation}
Moreover,
$
\Gal(H/L)=E[8],$
$\Gal(L/K)=\langle\tau\rangle .
$

\begin{proposition}
	\label{lem:strong-commutator}
	There exists a central commutator
	\[
	c\in\Gal(A/K)
	\]
	whose restriction to $F_{\mathrm{MW}}$ is the identity and whose
	restriction to $H$ is equal to $2P$.
\end{proposition}

\begin{proof}
	Let	$J=H\cap F_{\mathrm{MW}}$. Since $\Gal(F_{\mathrm{MW}}/L)$ is an elementary abelian
	$2$-group, the element $2Q\in\Gal(H/L)$ acts trivially on $J$.
	Therefore, the compatible pair	$(2Q,\operatorname{id})$
	defines an element $ s\in\Gal(A/L)$.
	
	Let $\widetilde{\tau}\in\Gal(A/K)$	be any lift of $\tau\in\Gal(L/K)$, and define
$
	c=\widetilde{\tau}s\widetilde{\tau}^{-1}s^{-1}.
$
	It is immediate that $c|_{F_{\mathrm{MW}}}=1$.
	On the other hand, restricting to $H$, we obtain
	\[
	c|_H
	=
	(\widetilde{\tau}s\widetilde{\tau}^{-1})|_H(s^{-1})|_H
	=
	(2P+2Q)-2Q
	=
	2P.
	\]
Finally, let $g\in \Gal(A/K)$. On $F_{\mathrm{MW}}$, we have
\[
(gcg^{-1})|_{F_{\mathrm{MW}}}
=
\operatorname{id}
=
c|_{F_{\mathrm{MW}}}.
\]
On the other hand, restricting to $H$, since $2P$ is $K$-rational, we have
\[
(gcg^{-1})|_H=c|_H=2P.
\]
Therefore, $gcg^{-1}=c$, and hence $c$ is central in
$\Gal(A/K)$. This completes the proof.
\end{proof}


\section{Construction of Special Cohomology Classes}
As in \eqref{aa}, we have the following exact sequence
\[
0\longrightarrow \langle P\rangle\simeq \mathbb{Z}/8\mathbb{Z}
\xrightarrow{i} E[8]
\xrightarrow{\pi} \mu_8
\longrightarrow 0,
\tag{$\mathcal E$}\label{eq:torsion-extension}
\]
induced by the Weil pairing $e_8(P,-)$. As in \eqref{bb}, together
with the Kummer sequence of $E$, we have the following commutative
diagram with exact row:
\begin{equation}\label{dd}
	\begin{tikzcd}
		&E(K)/8E(K)\ar[d,"\kappa_8"]&&\\
		H^1(K,\mathbb{Z}/8\mathbb{Z})
		\ar[r,"i_*"]
		&H^1(K,E[8])
		\ar[r,"\pi_*"]
		\ar[d,"\rho"]
		&H^1(K,\mu_8)\simeq K^*/K^{*8}
		\ar[r,"\delta"]
		&H^2(K,\mathbb{Z}/8\mathbb{Z})\\
		&H^1(K,E(\overline{K}))[8]&&
	\end{tikzcd}
\end{equation}

We choose a set of representatives for the finite subset
\[
\pi_*\circ\kappa_8(E(K)/8E(K))
\subset K^*/K^{*8},
\]
and denote it by $M$.

Choose a finite set of non-Archimedean  places $S$ containing:
\begin{enumerate}[label=\textup{(\roman*)}]
	\item all places of $K$ lying above $2$;
	\item all places of bad reduction of $E/K$;
	\item all places ramified in $L/K$;
	\item all places appearing in the elements of $M$.
\end{enumerate}
The field $K=\Q(\zeta_8)$ has class number one. Choose a ray modulus
$\mathfrak m$ supported on $S$, with local exponents sufficiently large such
that
\begin{equation}\label{eq:ray-local-power}
 x\equiv1\pmod{\mathfrak m}
 \quad\Longrightarrow\quad
 x\in K_v^{\times8}
 \quad(v\in S).
\end{equation}
Let $R_{\mathfrak m}/K$ be the corresponding ray class field.  Since the element $c$ of
Lemma~\ref{lem:strong-commutator} is a commutator, it acts trivially on the
abelian extension $ A\cap R_{\mathfrak m}$ of $K$.  Therefore $(c,1)$ defines an element of $\Gal(AR_{\mathfrak m}/K)$.

Since \(c\) is central, the Chebotarev density theorem yields infinitely many
places \(v\notin S\), unramified in \(AR_{\mathfrak m}\), such that the
Frobenius element \(\operatorname{Frob}_v\in \Gal(AR_{\mathfrak m}/K)\)
satisfies
\begin{equation}\label{eqe}
	\operatorname{Frob}_v|_H=2P,
	\qquad
	\operatorname{Frob}_v|_{F_{\mathrm{MW}}}=1,
	\qquad
	\operatorname{Frob}_v|_{R_{\mathfrak m}}=1.
\end{equation}
In particular, \(v\) splits completely in both \(F_{\mathrm{MW}}\) and
\(L\).

 From now on, we fix one such place $v$. Let \(\mathfrak{p}\) be the prime ideal corresponding to the place \(v\). The last condition in \eqref{eqe} says, by Artin reciprocity
for the ray class field $R_{\mathfrak m}/K$, that $\mathfrak p$ is trivial in
the ray class group.  Hence there is
$a\in K^\times$ such that
\begin{equation}\label{eq:ray-generator}
 (a)=\mathfrak p,
 \qquad
 a\equiv1\pmod{\mathfrak m}.
\end{equation}
Thus $v_{\mathfrak p}(a)=1$, the element $a$ has no other non-Archimedean
place, and it is an eighth power in $K_u$ for every $u\in S$.

\begin{lemma}\label{lem:a-norm}
The element $a$ belongs to $N_{L/K}(L^\times)$.  Moreover, its corresponding Kummer
class in $H^1(K,\mu_8)$ lifts through \eqref{dd} to a
class
$
 \eta_0\in H^1(K,E[8]).
$
\end{lemma}

\begin{proof}
It suffices to show that the element $a$ is a local norm everywhere.
\begin{enumerate}
	\item At the place $v$, the extension $L/K$ splits completely,
	because
$
	\operatorname{Frob}_{v}|_{F_{\mathrm{MW}}}=1.
$
	
	\item At every place $u\in S$, the element $a$ is an eighth power.
	Hence, it is a norm from $L_u/K_u$.
	
	\item Away from
	$S\cup\{v\}$, the extension is unramified and $a$ is a unit.
	The norm map on the units of an unramified local extension is
	surjective.
\end{enumerate}

By Proposition \ref{connect}, the connecting homomorphism $\delta$ in \eqref{eq:torsion-extension} is given by the cup product with the
character $c$ corresponding to the extension $L/K$. By the norm
criterion for the Hilbert symbol \cite[Chapter XIV, Corollary 1]{Serre}, together with the fact
that $a\in N_{L/K}(L^\times)$, we obtain $\delta(a)=0$. Therefore, by exactness, $a$ admits the required lift.
\end{proof}

We shall need a lift with controlled ramification.

\begin{lemma}\label{lem:unramified-lift}
	The lift $\eta_0$ can be chosen such that it is unramified at all  non-Archimedean  places  outside 
 $S\cup\{ v\}$.
\end{lemma}

\begin{proof}
	Let$
	\mathcal U
	=
	\operatorname{Spec}
	\mathcal O_K\left[1/(S\cup\{v\})\right].$
	The Kummer sequence in étale cohomology gives an exact sequence
	\[
	0\longrightarrow \Pic(\mathcal U)/8
	\longrightarrow H^2(\mathcal U,\mu_8)
	\longrightarrow \Br(\mathcal U)[8]
	\longrightarrow 0.
	\]
	Moreover, $\Pic(\mathcal U)=0$, since $K$ has class number one.
	The scheme $\mathcal U$ is regular and integral, and hence the natural
	map $
	\Br(\mathcal U)\longrightarrow \Br(K)$
	is injective \cite[Theorem 3.5.5]{CTSBr}. Therefore, the map $
	H^2(\mathcal U,\mu_8)\longrightarrow H^2(K,\mu_8)$
	is injective.
	
	After removing the places above $2$ and the places of bad reduction,
	the elliptic curve $E$ extends to an elliptic scheme over $\mathcal U$,
	the group scheme $E[8]$ is finite étale, and the $8$-torsion section
	$P$ together with the Weil-pairing quotient extend over $\mathcal U$.
	Consequently, the exact sequence
	\eqref{eq:torsion-extension} extends to a short exact sequence of finite
	étale group schemes over $\mathcal U$. Hence we obtain the following
	commutative diagram:
	\begin{equation}
		\begin{tikzcd}
			H^1(\mathcal U,E[8])\ar[r]\ar[d]
			&
			H^1(\mathcal U,\mu_8)\ar[r]\ar[d]
			&
			H^2(\mathcal U,\mathbb Z/8\mathbb Z)\ar[d]
			\\
			H^1(K,E[8])
			\ar[r,"\pi_*"]
			&
			H^1(K,\mu_8)\simeq K^*/K^{*8}
			\ar[r,"\delta"]
			&
			H^2(K,\mathbb Z/8\mathbb Z).
		\end{tikzcd}
	\end{equation}
	
	By the preceding discussion, the last vertical arrow is injective.
	Since $(a)=\mathfrak p$ and $\mathfrak p$ is inverted in $\mathcal U$,
	the element $a$ is a unit in $\mathcal O_{\mathcal U}$ and hence defines
	a class in
$	H^1(\mathcal U,\mu_8)$.
	Its image under the boundary map in
$	H^2(K,\mathbb Z/8\mathbb Z)$
	vanishes by Lemma~\ref{lem:a-norm}. By injectivity of the last vertical
	map, its boundary already vanishes in
$	H^2(\mathcal U,\mathbb Z/8\mathbb Z).$
	Therefore, by exactness, the class of $a$ admits a lift in
$	H^1(\mathcal U,E[8]).$
	This lift is unramified at every  place outside
	$S\cup\{v\}$.
\end{proof}
\subsection{Calculation of period-index obstruction map}
In summary, we have constructed infinitely many non-Archimedean places \(v\notin S\)
such that
\begin{equation}\label{eq:Frob-conditions}
	\operatorname{Frob}_v|_H=2P,
	\qquad
	\operatorname{Frob}_v|_{F_{\mathrm{MW}}}=1,
	\qquad
	\operatorname{Frob}_v|_{R_{\mathfrak m}}=1.
\end{equation}
Fix such a place \(v\), and let \(\mathfrak p\) be the prime ideal
corresponding to \(v\). Write \(\mathfrak p=(a)\). By Lemma
\ref{lem:unramified-lift}, \(a\) admits a lift $
\eta_0\in H^1(K,E[8]),$
which is unramified at every non-Archimedean place outside
\(S\cup\{v\}\). Put $\xi=\kappa_8(P)\in H^1(K,E[8]),$ and let $\inv_v: \Br(K_v)\rightarrow \mathbb{Q}/\mathbb{Z}$ denote the local invariant map.
 \begin{theorem}\label{cc}
 	$\inv_{v}T(\eta_0,P)=-\frac{2}{8}$.
 \end{theorem}
    
    Before proving this theorem, we make some preparations. Let
    $K_v$ be the completion of $K$ at the place $v$, and let
    $G_{K_{v}}$ denote the decomposition group of $G_K$ at
    $v$. For convenience, we adopt the convention that, for any cohomology class \(\eta\), \(\eta_{v}\) denotes its restriction to \(K_{v}\).
    
    At the place $v$, the extension $L/K$ splits completely.
    Therefore, $G_{K_{v}}$ acts trivially on $E[8]$. Hence, we may
    view the corresponding cohomology classes as homomorphisms $
    f,g\in\operatorname{Hom}(G_{K_{v}},E[8])  $
    representing $(\eta_0)_{v}$ and $\xi_{v}$, respectively.  Moreover, $f$ is a tame character because $v\nmid 2$, while $g$ is unramified by
    \cite[Chapter VIII, Proposition 2.1]{Sil}. By \eqref{eq:Frob-conditions}, we have
    $\label{eq:z-frob}
    	g(\operatorname{Frob}_{v})=2P.
    $
Recall that we have fixed a symplectic basis $(P,Q)$ satisfying
$e(P,Q)=\zeta_8$. We may then write
\[
f(\sigma)=a(\sigma)P+b(\sigma)Q,\qquad
g(\sigma)=c(\sigma)P+d(\sigma)Q,
\]
where
\[
a,b,c,d\in
\operatorname{Hom}_{\mathrm{cont}}
(G_{K_{v}},\mathbb Z/8\mathbb Z).
\]
We are now ready to prove  Theorem \ref{cc}.
\begin{proof}
	Firstly, we compute the cup product. For
	$\sigma,\tau\in G_{K_{v}}$, we have
	\begin{align*}
		(f\cup g)(\sigma,\tau)
		&=e_8(f(\sigma),\sigma(g(\tau)))\\
		&=e_8(f(\sigma),g(\tau))\\
		&=e_8(a(\sigma)P+b(\sigma)Q,
		c(\tau)P+d(\tau)Q)\\
		&=\zeta_8^{a(\sigma)d(\tau)-b(\sigma)c(\tau)}.
	\end{align*}
	
	By the definition of $\pi$, we have $	\pi_*(g)(\sigma)=\zeta_8^{d(\sigma)}$.
	On the one hand, by the choice of $S$, the class $\pi_*(g)$ is trivial in
	$H^1(K_{v},\mu_8)$. Hence, $d$ is the trivial character.
	Therefore,	$	g(\sigma)=c(\sigma)P$,
	and consequently, $
	(f\cup g)(\sigma,\tau)
	=
	\zeta_8^{-b(\sigma)c(\tau)}$.
	On the other hand, we have
	$
	\pi_*(f)(\sigma)
	=
	e(P,a(\sigma)P+b(\sigma)Q)
	=
	\zeta_8^{b(\sigma)}.
	$
	Moreover, by the construction of $\eta_0$, the character $\pi_*(f)$
	corresponds to the element $a\in K_{v}^{\times}/K_{v}^{\times 8}$
	with $v_{\mathfrak p}(a)=1$.
	Therefore, by \cite[Chapter XIV, Proposition 3]{Serre}, the local invariant
	of $f\cup g$ is given by
	\[
	\operatorname{inv}_{v}(f\cup g)
	=
	-\frac{c(\operatorname{Frob}_{v})}{8}.
	\]
	Using \eqref{eq:z-frob}, we obtain
$
	\operatorname{inv}_{v}(f\cup g)
	=
	-\frac{2}{8}.
$
\end{proof}

We next modify $\eta_0$ by a class from the kernel of
\eqref{eq:torsion-extension} so as to impose the equality needed for the
index calculation.  Choose representatives in $K^\times$ for all Kummer
classes occurring below.  In particular, write
\[
 z=\pi_*(\xi)\in H^1(K,\mu_8)=K^\times/K^{\times8}
\]
and define
\begin{equation*}\label{eq:defect-t-M}
 \delta_0=T(\eta_0,\xi-\eta_0),
 \qquad
 t=\frac z{a^2},
 \qquad
 N=K(\sqrt[8]{t}).
\end{equation*}

\begin{lemma}\label{lem:defect-relative}
The element $\delta_0$ lies in the relative Brauer group
\[
 \Br(N/K)=\ker\bigl(\Br(K)\to\Br(N)\bigr).
\]
\end{lemma}

\begin{proof}
By the Brauer-Hasse-Noether theorem, it is sufficient to show that
$\delta_0$ is trivial at every place of $N$. Let $w$ be a place of $N$,
and let $u$ be the place of $K$ lying below $w$.

If $u\notin S\cup\{v\}$, then both $\eta_0$ and $\xi$ are
unramified at $u$. Their cup product is inflated from a class over the residue
field, which has cohomological dimension one. Hence, $
(\delta_0)_w=0$.

If $u\in S$, then $a$ is an eighth power in $K_u$. By the exactness of
\eqref{dd}, the class $(\eta_0)_u$ lies in the image of
\[
H^1(K_u,\langle P\rangle)\longrightarrow H^1(K_u,E[8]).
\]
Over every completion $N_w$, both $z$ and $a$ become eighth powers in
$N_w$. Hence, $\xi-\eta_0$ also lies in the image of
\[
H^1(N_w,\langle P\rangle)\longrightarrow H^1(N_w,E[8]).
\]
Since $\langle P\rangle$ is isotropic with respect to the Weil pairing
$e_8$, we have
$
\operatorname{res}_{N_w}(\delta_0)_u=0.
$

Now suppose that $u=v$. $\delta_0=T(\eta_0,\xi-\eta_0)=T(\eta_0,\xi)-T(\eta_0,\eta_0)$. Let $w$ be a place of $N$ above
$v$. By the choice of $S$, the class $\pi_*(\xi)$ is trivial,
so $z$ is an eighth power locally at $v$. Since $v_{\mathfrak p}(a)=1$,
the local Kummer class of $t=z/a^2$
has exact order \(4\). Therefore, for every \(w\mid v\), $[N_w:K_{v}]=4.
$

By Theorem \ref{cc}, the first term \(T(\eta_0,\xi)\) has order \(4\). By Proposition \ref{ff}, the
second term $T(\eta_0,\eta_0)=2Ob(\eta_0)$, which is also annihilated by \(4\). Hence
\((\delta_0)_v\) is annihilated by \(4\). Thus by \cite[Chapter XIII, Proposition 6]{Serre}, $\operatorname{res}_{N_w/K_{v}}$ is trivial.
\end{proof}


\begin{lemma}
\label{lem:relative-symbol}
Every class $\gamma\in\Br(N/K)$ can be written $ \gamma=[\chi,t]_8 $ for some $\chi\in H^1(K,\mathbb{Z}/8\mathbb{Z})$.
\end{lemma}
\begin{proof}
This is \cite[Section 4.7, Corollary 4.7.6]{GilleSzamuely}.
\end{proof}

Applying Lemma~\ref{lem:relative-symbol} to $-\delta_0$, there exists $\chi\in H^1(K,\mathbb{Z}/8\mathbb{Z})$ such that
 $\label{11}
 	 [\chi,t]_8=-\delta_0.
$

Set $\label{eq:eta-final}
 \eta=\eta_0+i_*(\chi)$.
By Proposition \ref{tate}, for every \(y\in H^1(K,E[8])\), we have $
T(i_*(\chi),y)=[\chi,\pi_*(y)]_8.$
We recall the notation
\[
\kappa_8(P)=\xi,\qquad z=\pi_*(\xi),\qquad
a=\pi_*(\eta_0),\qquad \delta_0=T(\eta_0,\xi-\eta_0),
\qquad
t=\frac z{a^2}.
\]
We now expand the terms one by one:
\begin{align*}
 T(\eta,\xi-\eta)
 &=T(\eta_0+i_*(\chi),\xi-\eta_0-i_*(\chi))\notag\\
 &=\delta_0+T(i_*(\chi),\xi)-T(i_*(\chi),\eta_0)
   -T(\eta_0,i_*(\chi))-T(i_*(\chi),i_*(\chi))\notag\\
 &=\delta_0+[\chi,z]_8-2[\chi,a]_8   \text{ (by Proposition \ref{tate})}\notag\\
 &=\delta_0+[\chi,z/a^2]_8\notag\\
 &=0.\label{eq:defect-killed} \qquad \text{ (by  \eqref{11})}
\end{align*}
Here $T(i_*(\chi),i_*(\chi))=0$ by isotropy.

Let $[C]\in H^1(K,E(\overline{K}))$ be the corresponding homogeneous space of $\eta$, and put
\begin{equation*}\label{eq:alpha-beta}
 \alpha=Ob(\eta),
 \qquad
 \beta=T(\eta, P)\overset{\mathrm{def}}{:=}T(\eta,\xi).
\end{equation*}
Since $T(\eta,\xi-\eta)=T(\eta,\xi)-T(\eta,\eta)$, we obtain
\begin{equation}\label{eq:beta-2alpha}
   \beta=2\alpha.
\end{equation}

At $v$, both $i_*(\chi)$ and $\xi$ take values in the
isotropic subgroup $\langle P\rangle$, thus $\beta_v=T(\eta,\xi)_v=T(\eta_0,\xi)_v+T(i_*(\chi),P)_v=T(\eta_0,\xi)_v$.  Therefore, by Theorem \ref{cc} we have$\label{eq:beta-final-local}
 \inv_{v}(\beta)=-\frac28.$
In particular,  $\alpha_{v}$  has  order exactly $8$. 
\section{Proof  of Theorem \ref{mainthm}}

In Section $3$, we construct $\eta\in H^1(K,E[8])$ and denote the corresponding homogeneous space by $C$.	Recall the following Kummer sequence:\label{tt}
 \begin{align}\label{kummer}
 		0\rightarrow E(K)/8E(K)\xrightarrow{\kappa_8} &H^1(K,E[8])\xrightarrow{\rho} H^1(K,E(\overline{K}))[8]\rightarrow 0.
 \end{align}
By \eqref{eq:beta-2alpha}, we have $T(P,\eta)=2Ob_8(\eta)$ and $Ob_8(\eta)_{v}$ has exact order $8$. 
\begin{theorem}\label{22}
	The curve \(C\) has period \(8\) and index \(16\). Furthermore, for every lift \(\gamma\in H^1(K,E[8])\) of  \([C]\), the obstruction class \(Ob_8(\gamma)\) has  order exactly \(8\). 
\end{theorem}
\begin{proof}
	For any lift $\gamma\in H^1(K,E[8])$ of $[C]$, we first prove that
	$Ob_8(\gamma)$ has  order  exactly $8$.
	
	By the exactness of \eqref{kummer}, there exists
	$S\in E(K)$ such that $\gamma=\eta+\kappa_8(S).$
	By \eqref{ff}, we have
	\[
	\begin{aligned}
		Ob_8(\gamma)
		&=Ob_8(\eta+\kappa_8(S))\\
		&=Ob_8(\eta)+Ob_8(\kappa_8(S))
		+T(S,\eta)\\
		&=Ob_8(\eta)+T(S,\eta)
	\end{aligned}
	\]
	, where we use the fact that $Ob_8(\kappa_8(S))=0$.
	
	Restricting to $K_{v}$, the class
	$Ob_8(\eta)_{v}$ has exact order $8$. On the other hand, by the
	definition of $F_{\mathrm{MW}}$ and the condition $\operatorname{Frob}_{v}|_{F_{\mathrm{MW}}}=1$,
	the image of $E(K)$ in $E(K_{v})$ is contained in
	$2E(K_{v})$. Hence, the order of $	T(S,\eta)_{v}$
	divides \(4\).
	
 Consequently,	$Ob_8(\gamma)_{v}$
	also has exact order \(8\), and hence so does \(Ob_8(\gamma)\).

Now we prove that $P(C)=8$. First, it is clear that
the period of $C$ divides $8$. Suppose, for contradiction, that
$[C]\in H^1(K,E(\overline{K}))[4]$. By taking $n=4$, $m=2$ in  Theorem \ref{cd}, there exists a lift
$
\gamma\in H^1(K,E[8])
$
of $C$ such that
$
Ob_8(\gamma)\in \Br(K)[2].
$
This contradicts the fact proved above.
 Finally, we compute the index of $C$. Let
 \[
 j_*:H^1(K,E[8])\longrightarrow H^1(K,E[16])
 \]
 be the map induced by the natural inclusion $ E[8]\hookrightarrow E[16]$.
 Then
 \[
 \begin{aligned}
 	Ob_{16}(j_*(\eta)-\kappa_{16}(P))
 	&=
 	Ob_{16}(j_*(\eta))
 	-T(j_*(\eta),\kappa_{16}(P))\\
 	&=
 	2Ob_8(\eta)-T(\eta,\kappa_8(P))
 	\qquad\text{(by Proposition \ref{tate2})}\\
 	&=0.   \qquad \text{(by \eqref{eq:beta-2alpha})}
 \end{aligned}
 \]
 Here,
$
 \kappa_{16}:E(K)/16E(K)\longrightarrow H^1(K,E[16])
$
 denotes the connecting homomorphism. Therefore,
 by the same argument as in Proposition \ref{neil}, the index of \(C\)
 divides \(16\).
 
 Moreover, the index of $C$ cannot be equal to $8$. Assume that $I(C)=8$, by Theorem $5$ in \cite{clark},  there exists
 lift $\gamma\in H^1(K,E[8])$ of $C$ such that
$
 Ob_8(\gamma)=0,
$ which is a contradiction.
 Consequently, the index of $C$ is exactly $16$.
\end{proof}

For the fixed curve $E=E_\lambda$, the Chebotarev density theorem
provides infinitely many  places $v$. We show that
the homogeneous spaces $C_{v}$ constructed in Theorem \ref{22} are
pairwise distinct.

Let $a_{v}$ be the ray generator defined in
\eqref{eq:ray-generator}. Suppose, for contradiction, that
$
C_{v}=C_{v'}
$
for two distinct places $v$ and $v'$. Then their
chosen lifts differ by a global Kummer class. Applying the quotient map
$\pi_*$ associated with \eqref{dd}, we obtain
\[
\frac{a_{v}}{a_{v'}}
\in
\pi_*\kappa_8(E(K))+K^{\times 8}
\quad\text{in }K^\times/K^{\times8}.
\]

By the choice of $S$, representatives of the finite group
$\pi_*\kappa_8(E(K))$ can be chosen whose primes are all contained
in $S$. However, the left-hand side has valuation \(1\) at
$v$. This is impossible in
$K^\times/K^{\times8}$. Therefore, the homogeneous spaces
$C_{v}$ are pairwise distinct.

In Section 3, we construct infinitely many homogeneous spaces that provide counterexamples on each curve \(E_\lambda\) with \(\lambda\in\Omega\). In Section 2, we construct infinitely many such elliptic curves with pairwise distinct \(j\)-invariants. This completes the proof of Theorem~\ref{mainthm}.

\appendix
\makeatletter
\renewcommand{\subsection}{%
	\@startsection{subsection}{2}%
	{\z@}%
	{.5\linespacing\@plus.7\linespacing}%
	{.5\linespacing}%
	{\normalfont\bfseries}%
}
\makeatother

\section{Exact Magma certificates}

The calculations below were performed with Magma V2.29-9.  They use exact
number fields and exact finite fields; no floating-point computation enters the proof.

\subsection{The full division field and the action of $\tau$}\label{app:magma}

\begin{lstlisting}[style=magma]
	
	Q0 := Rationals();
	t := Q0!(-17)/2;
	b := (2*t - 1)*(t - 1);
	c := b/t;
	
	E := EllipticCurve([Q0 | 1-c, -b, -b, 0, 0]);
	
	K<zeta> := CyclotomicField(8);
	R<x> := PolynomialRing(K);
	EK := BaseChange(E,K);
	
	psi8full, psi8, psi2 := DivisionPolynomial(EK,8);
	
	
	facK := Factorization(psi8full);
	
	print "factor degrees over K",
	[ <Degree(f[1]),f[2]> : f in facK ];
	
	print "x^8 - 647 irreducible over K:",
	IsIrreducible(x^8 - K!647);
	
	L<u> := ext<K | x^8 - 647>;
	EL := BaseChange(E,L);
	
	RL<X> := PolynomialRing(L);
	ff := Factorization(RL!psi8full);
	
	print "factorization over L",
	#ff,
	[ <Degree(g[1]),g[2]> : g in ff ];
	
	roots := [
	-Coefficient(g[1],0)/Coefficient(g[1],1)
	: g in ff
	| Degree(g[1]) eq 1
	];
	
	aa := aInvariants(EL);
	pts := [];
	numberSquareDiscriminants := 0;
	
	for xx in roots do
	
	By := aa[1]*xx + aa[3];
	
	dy :=
	By^2
	+ 4*(
	xx^3
	+ aa[2]*xx^2
	+ aa[4]*xx
	+ aa[5]
	);
	
	ok, sy := IsSquare(dy);
	
	if ok then
	
	numberSquareDiscriminants +:= 1;
	
	y1 := (-By + sy)/2;
	y2 := (-By - sy)/2;
	
	Append(~pts, EL![xx,y1,1]);
	
	if sy ne 0 then
	Append(~pts, EL![xx,y2,1]);
	end if;
	
	end if;
	
	end for;
	
	print "number of x-roots", #roots;
	print "square y-discriminants",
	numberSquareDiscriminants;
	print "number of distinct nonzero E[8]-points",
	#pts;
	
	P := EL![0,0,1];
	
	print "P order", Order(P);
	
	Qbasis := EL!0;
	
	for QQ in pts do
	if Order(QQ) eq 8
	and WeilPairing(P,QQ,8) eq L!zeta then
	
	Qbasis := QQ;
	break;
	
	end if;
	end for;
	
	print "symplectic complement found", Qbasis ne EL!0;
	
	emb := hom<K -> L | L!zeta>;
	tau := hom<L -> L | emb, (L!zeta)*u>;
	
	tauQ := EL![tau(Qbasis[1]), tau(Qbasis[2]), tau(Qbasis[3])];
	
	DD := tauQ - Qbasis;
	cc := -1;
	
	for j in [0..7] do
	if DD eq j*P then
	cc := j;
	break;
	end if;
	end for;
	
	print "tau(Q)-Q = cP with c=", cc;
\end{lstlisting}

The expected decisive output is of the following form.

\begin{lstlisting}[style=magma]
	factor degrees over K [ <1, 1>, <1, 1>, <1, 1>, <1, 1>, <1, 1>, <2, 1>, <2, 1>,
	<4, 1>, <4, 1>, <8, 1>, <8, 1> ]
	x^8 - 647 irreducible over K: true
	factorization over L 33 [ <1, 1>, <1, 1>, <1, 1>, <1, 1>, <1, 1>, <1, 1>, <1,
	1>, <1, 1>, <1, 1>, <1, 1>, <1, 1>, <1, 1>, <1, 1>, <1, 1>, <1, 1>, <1, 1>, <1,
	1>, <1, 1>, <1, 1>, <1, 1>, <1, 1>, <1, 1>, <1, 1>, <1, 1>, <1, 1>, <1, 1>, <1,
	1>, <1, 1>, <1, 1>, <1, 1>, <1, 1>, <1, 1>, <1, 1> ]
	number of x-roots 33
	square y-discriminants 33
	number of distinct nonzero E[8]-points 63
	P order 8
	symplectic complement found true
	tau(Q)-Q = cP with c = 1
\end{lstlisting}

\subsection{The Kummer Frobenius at 281}\label{magma2}

\begin{lstlisting}[style=magma]
	procedure LocalKappaAllEmbeddings(p)
	
	f := Order(Integers(8)!p);
	qres := p^f;
	
	F<g> := GF(qres);
	
	t := F!(-17)/2;
	b := (2*t - 1)*(t - 1);
	c := b/t;
	
	E := EllipticCurve([
	F | 1-c, -b, -b, 0, 0
	]);
	
	O := E!0;
	P := E![0,0,1];
	
	zeta := RootOfUnity(8,F);
	
	tor := DivisionPoints(O,8);
	
	Q := O;
	
	for S in tor do
	if Order(S) eq 8
	and WeilPairing(P,S,8) eq zeta then
	
	Q := S;
	break;
	
	end if;
	end for;
	
	print "symplectic complement found", Q ne E!0;
	
	F8<w> := ext<F | 8>;
	E8 := BaseChange(E,F8);
	
	PP := E8!P;
	divs := DivisionPoints(PP,8);
	
	if #divs eq 0 then
	print "No division point was found";
	return;
	end if;
	
	R := divs[1];
	
	FrobR := E8![
	R[1]^qres,
	R[2]^qres,
	R[3]^qres
	];
	
	T := FrobR - R;
	
	ispow, eighthroot := IsPower(F!647,8);
	
	print "prime, residue degree, number of K-primes",
	p, f, 4 div f;
	
	print "647 is an eighth power",
	ispow, eighthroot;
	
	for j in [1,3,5,7] do
	
	Qj := E8!(j*Q);
	
	u := -1;
	v := -1;
	
	for ii in [0..7] do
	for jj in [0..7] do
	
	if T eq ii*PP + jj*Qj then
	u := ii;
	v := jj;
	end if;
	
	end for;
	end for;
	
	print <j,zeta^j,u,v>;
	
	end for;
	
	end procedure;
	
	LocalKappaAllEmbeddings(281);
\end{lstlisting}

The expected output is

\begin{lstlisting}[style=magma]
	symplectic complement found true
	prime, residue degree, number of K-primes 281 1 4
	647 is an eighth power true 63
	<1, 89, 5, 3>
	<3, 221, 5, 1>
	<5, 192, 5, 7>
	<7, 60, 5, 5>
\end{lstlisting}

\end{document}